\documentclass{amsart}

\usepackage{amssymb}
\usepackage{centernot}
\usepackage{hyperref}
\usepackage[english]{babel}

\usepackage{blindtext}
\usepackage{lscape}
\usepackage[figuresright]{rotating}
\usepackage{printlen}
\usepackage{mdframed}
\usepackage{tikz}
\usepackage{caption}
\usetikzlibrary{arrows,positioning,chains,matrix,scopes}

\theoremstyle{plain}
\newtheorem{thm}{Theorem}[section]
\numberwithin{thm}{section}
\newtheorem{prob}[thm]{Problem}
\newtheorem{fact}[thm]{Fact}
\newtheorem{cor}[thm]{Corollary}
\newtheorem{lemma}[thm]{Lemma}
\newtheorem{prop}[thm]{Proposition}
\theoremstyle{definition}
\newtheorem{defn}[thm]{Definition}

\theoremstyle{remark}

\newtheorem{claim}[thm]{Claim}

\newcommand{\ms}{\mathsf{s}}
\newcommand{\G}{\mathsf{G}}

\newcommand{\D}{\mathrm{D}}

\newcommand{\sfin}{\mathsf{S}_\mathrm{fin}}
\newcommand{\gfin}{\mathsf{G}_\mathrm{fin}}
\newcommand{\sbound}{\mathsf{S}_\mathrm{bnd}}
\newcommand{\gbound}{\mathsf{G}_\mathrm{bnd}}

\newcommand{\fin}{\mathrm{fin}}

\DeclareMathOperator{\impls}{\implies}

\DeclareMathOperator{\dom}{dom}
\DeclareMathOperator{\modd}{mod}
\DeclareMathOperator{\wins}{\uparrow}
\DeclareMathOperator{\doesntwin}{\centernot{\uparrow}}
\DeclareMathOperator{\restrict}{\upharpoonright}

\newcommand{\mA}{{\mathcal{A}}}
\newcommand{\mB}{{\mathcal{B}}}
\newcommand{\mC}{{\mathcal{C}}}

\newcommand{\mF}{{\mathcal{F}}}

\newcommand{\mO}{{\mathcal{O}}}

\newcommand{\mU}{{\mathcal{U}}}
\newcommand{\mV}{{\mathcal{V}}}
\newcommand{\mW}{{\mathcal{W}}}

\newcommand{\NN}{\mathbb{N}}
\newcommand{\RR}{\mathbb{R}}

\newcommand{\w}{{\omega}}
\newcommand{\W}{{\Omega}}

\newcommand{\1}{\textsc{Alice}}
\newcommand{\A}{\textsc{A}}
\newcommand{\2}{\textsc{Bob}}
\newcommand{\B}{\textsc{B}}
\newcommand{\seq}[1]{{\langle {#1} \rangle}}
\newcommand{\set}[1]{\left\{\, {#1} \,\right\}}
\newcommand{\sel}[3]{\mathsf{S}_\mathrm{#1}(#3)\modd{#2}}
\newcommand{\sels}[3]{\mathsf{S}^\ms_\mathrm{#1}(#3)\modd{#2}}
\newcommand{\ssel}[2]{\mathsf{S}_\mathrm{#1}(#2)}
\newcommand{\game}[3]{\mathsf{G}_\mathrm{#1}(#3)\modd{#2}}
\newcommand{\sgame}[2]{\mathsf{G}_\mathrm{#1}(#2)}

\begin{document}
	
\title{Topological games of bounded selections}
\author[L. Aurichi]{Leandro F. Aurichi$^1$}
\address{Instituto de Ci\^encias Matem\'aticas e de Computa\c c\~ao, Universidade de S\~ao Paulo\\
	Avenida Trabalhador s\~ao-carlense, 400,  S\~ao Carlos, SP, 13566-590, Brazil}
\email{aurichi@icmc.usp.br}
\thanks{$^1$Supported by FAPESP (2017/09252-3)}
\author[M. Duzi]{Matheus Duzi$^2$}
\address{Instituto de Ci\^encias Matem\'aticas e de Computa\c c\~ao, Universidade de S\~ao Paulo\\
	Avenida Trabalhador s\~ao-carlense, 400,  S\~ao Carlos, SP, 13566-590, Brazil}
\email{matheus.duzi.costa@usp.br}
\thanks{$^2$Supported by FAPESP (2017/09797-0)}
\begin{abstract}
	We present a new variation of the classical selection principles $\ssel{k}{\mA,\mB}$ ($k\in\NN$) and $\ssel{fin}{\mA,\mB}$ that formally lies between these two properties. As in the case of the classical selection principles, we also obtain a new variation of topological game and discuss how new topological properties may emerge in the specific cases of covering and tightness.
\end{abstract}
\maketitle

\section{Introduction}

Topological games have been studied for several years. Arguably the best known, and oldest, is the Banach-Mazur game (see e.g. \cite{Telgarsky1987}) but many others have been studied. Some classical properties defined by Menger, Rothberger and Hurewicz (see e.g. \cite{hurewicz1926}) are nowadays presented in a form of a game or selection principle - which has the advantage of showing precisely the combinatorics behind those properties. A selection principle (see below) usually is of the following form: a sequence of special sets is given, then one can pick elements of each set to form a new special set. Per example, if for every sequence of open coverings is possible to pick one open set from each covering to form a new open covering, we say that the space satisfies the Rothberger property. If it is allowed to pick not only one open set, but finitely many from each open covering, then we have the Menger property. The difference in the game versions of these selection principles is that one of the players gives each special set, one at a time - so the other player has to choose one (or finitely many) element of such sets without knowing which are the other special sets in order to form a new special set by the end of the game.

With the two examples above, it is easy to see that small changes in the statements can change drastically the final property. Per example, every compact space is (trivially) Menger, but not necessarily Rothberger. So it is natural to ask what other kind of change can be done. What happens when the selection is not only one element, but two? Or three? As it turns out, such changes might give rise to new properties. In \cite{Aurichi2018}, for instance, it was shown that when we consider the family of special sets as the one with every subset of a given space whose closure contains a fixed point of such space (the so called tightness case), then it actually makes a difference which amount of points we allow the second player to pick. On the other hand, in the covering case previously mentioned, it was shown in \cite{Crone2019} that, at least for Hausdorff spaces, it makes no difference which fixed amount of open sets we allow the second player to choose.

In this paper we give continuity in this study of the fundamental differences between the covering (as Rothberger and Menger) and the tightness (as in properties like countable fan tightness and countable strong fan tightness) cases. In order to do so we introduce a new kind of variation: what changes if each selection is finite but at the end, the size of all selections has to be bounded by a number? We will show that usually this bounded selections are different from the classical ones and that the behavior can also change depending upon the case (covering, tightness) studied, highlighting a few of what appears to be the reasons for this phenomenon. In the covering case, notably, we show some characterizations for the new game and selection principle variations analogous to classical ones and, as a corollary, we present a characterization for metrizable spaces in terms of two subspaces: a compact and a countable (or strong measure zero with respect to every metric that generates the space's topology). 

This paper is organized as follows. In Section \ref{SEC_Sel_Prin} we present the new variation of selection principle and discuss its first relations with some classical selection principles, showing that in the covering case we have a new intermediate property and that in the tightness case the new variation collapses to one of the classical variations.

In Section \ref{SEC_Assoc_Games} we present the games naturally associated to the new variation, showing that both in the covering and tightness cases we have new games. In particular, we characterize the new game in the tightness case in terms of the classical games.

We dedicate Section \ref{SEC_mod1_modfin} to present yet two other new variations of selection principles that enable us to characterize the covering case. 

In Section \ref{SEC_Paw_Hur} we show a result for the new variations in the covering case that is analogous to the Pawlikowski and Hurewicz theorems, obtaining yet another characterization of the new variation of selection principle. This result, however, could not be obtained as a corollary of the classical ones, so the proof is presented as an adaptation of the proof of Pawlikowski's Theorem, inspired by a simplified version seen in the notes \cite{Szewczak} provided by Szewczak and Tsaban.

We continue to study the covering case in Section \ref{SEC_dual}, where we present a duality analogous to the one given by Galvin in \cite{Galvin1978}.

Finally, Section \ref{SEC_CONC} is dedicated to present some known results and examples so we can summarize in two diagrams the contrast reflected by these bounded selections between the covering and tightness cases.

In what follows we denote $\w\setminus\{0\}$ by $\NN$. Also, given a topological space $X$ and $p\in X$, we write $\W_p=\set{Y\subset X: p\in\overline{Y}}$ and $\mO$ as the set of open covers of $X$. Given two open covers $\mU,\mW$ of a space, we denote their common refinement by $\mU\wedge\mW$, that is, 
\[\mU\wedge\mW=\set{U\cap W: U\in\mU, W\in\mW}.\]

The following trivial fact about topological spaces will also be useful for future arguments.

\begin{fact}\label{triv_fact0}
	Let $X$ be a topological space and $p\in X$. If $A$ is such that $p\in\overline{A}$ and $p\notin \overline{\{x\}}$ for every $x\in A$, then $p\in\overline{A\setminus F}$ for every $F\subset A$ finite.
\end{fact}

\section{Selection Principles}\label{SEC_Sel_Prin}
We formalize and fix the notation for the already discussed selection principle as it follows.
	
\begin{defn}\label{sk_defn}
	Let $\mA,\mB$ be families of sets and $k\in\NN$. We say that $\ssel{k}{\mA,\mB}$ holds if, for every sequence  $\seq{A_n:n\in\w}$ of elements of $\mA$ there is a sequence $\seq{B_n:n\in\w}$ such that
	\begin{enumerate}
		\item [a.] $B_n\subset A_n$ for every $n\in\w$;
		\item [b.] $|B_n|\le k$ for every $n\in\w$;
		\item [c.] $\bigcup_{n\in\w}B_n\in\mB$.
	\end{enumerate}
\end{defn}

We note that the classical case here is for $k=1$ in Definition \ref{sk_defn} with requirement $b$ changed to ``$|B_n|=1$ for every $n\in\w$''. For the two cases we are going to discuss in this paper, this small change makes no difference. It is worth mentioning that $\ssel{1}{\mO,\mO}$ is known as \emph{Rothberger property} and $\ssel{1}{\W_p, \W_p}$ is known as \emph{countable strong fan tightness}.

\begin{defn}\label{sfin_defn}
	Let $\mA,\mB$ be families of sets. We say that $\ssel{fin}{\mA,\mB}$ holds if, for every sequence  $\seq{A_n:n\in\w}$ of elements of $\mA$ there is a sequence $\seq{B_n:n\in\w}$ such that
	\begin{enumerate}
		\item [a.] $B_n\subset A_n$ is finite for every $n\in\w$;
		\item [b.] $\bigcup_{n\in\w}B_n\in\mB$.
	\end{enumerate}
\end{defn}

We note here that $\ssel{\fin}{\mO,\mO}$ is known as \emph{Menger property} and $\ssel{\fin}{\W_p, \W_p}$ is known as \emph{countable fan tightness}.

It is, then, based on Definitions \ref{sk_defn} and \ref{sfin_defn} that we define what may be a new variation: 
\begin{defn}\label{sbound_defn}
	Let $\mA,\mB$ be families of sets. We say that $\sbound(\mA,\mB)$ holds if, for every sequence  $\seq{A_n:n\in\w}$ of elements of $\mA$ there is a sequence $\seq{B_n:n\in\w}$ and $k\in\w$ with, for every $n\in\w$,
	\begin{enumerate}
		\item [a.] $B_n\subset A_n$ is finite;
		\item [b.] $\bigcup_{n\in\w}B_n\in\mB$;
		\item [c.] $|B_n|\le k$.
	\end{enumerate}
\end{defn}

It is easy to see that Definition \ref{sbound_defn} is different from both $\ssel{1}{\mA,\mB}$ and $\ssel{fin}{\mA,\mB}$:
\begin{prop}\label{covering_ex}
	$\sbound(\mO,\mO)$ holds over every compact space, but $\ssel{1}{\mO,\mO}$ does not hold over $2^\w$.\\
	Moreover, $\sfin(\mO,\mO)$ holds over every $\sigma$-compact space, but $\sbound(\mO,\mO)$ does not hold over $\RR$.
\end{prop}

On the other hand, for some choices of the families $\mA$ and $\mB$, the new definition may collapse to classical selection principles (the following result is somewhat of a generalization of Lemma 3.5 from \cite{Garcia-Ferreira1995}).
\begin{prop}\label{tight_eqv0}
	Let $(X,\tau)$ be a topological space and $p\in X$. Then the following properties are equivalent:
	\begin{itemize}
		\item[(1)] $\ssel{1}{\Omega_p, \Omega_p}$;
		\item[(2)] $\ssel{k}{\Omega_p, \Omega_p}$, for every $k\in\NN$;
		\item[(3)] $\ssel{bnd}{\Omega_p, \Omega_p}$.
	\end{itemize}
\end{prop}
\begin{proof}
	The implications $(1)\impls (2)\impls (3)$ are clear, so suppose $\sbound(\Omega_p, \Omega_p)$ holds and let $\seq{A_n:n\in\w}$ be a sequence of subsets of $X$ such that $p\in\overline{A_n}$ for every $n\in\w$. Since $\sbound(\Omega_p, \Omega_p)$ holds, there is a sequence $\seq{B_n:n\in\w}$ and $k\in\w$ with, for every $n\in\w$,
	\begin{enumerate}
		\item [a.] $B_n\subset A_n$;
		\item [b.] $p\in\overline{\bigcup_{n\in\w}B_n}$;
		\item [c.] $|B_n|\le k$.
	\end{enumerate}
	Without loss of generality, we may assume that $|B_n|= k$ for every $n\in\w$ and we write $B_n=\set{b_n^1, \dotsc, b_n^k}$ for each $n\in\w$. Now, let $C_i=\set{b_n^i:n\in\w}$ for each $i\le k$. 
	\begin{claim}
		There is an $i\le k$ such that $p\in \overline{C_i}$.
	\end{claim}
	\begin{proof}
		Just note that $\bigcup_{n\in\w}B_n=\bigcup_{i\le k}C_i$ and $\overline{\bigcup_{i\le k}C_i}=\bigcup_{i\le k}\overline{C_i}$.
	\end{proof}
	Let $m \le k$ be such that $p\in \overline{C_m}$. Then the sequence $\seq{b_n^m:n\in\w}$ witnesses $\ssel{1}{\Omega_p, \Omega_p}$ and the proof is complete.
\end{proof}

But even when the selection principle collapses, we may find new properties when looking into the new associated games. 

\section{The associated games}\label{SEC_Assoc_Games}
Just like we did with the selection principles, we formalize and fix a notation for the games that have already been discussed in the introduction as it follows:

\begin{defn}\label{k_game}
	Let $\mA,\mB$ be families of sets and $k\in\NN$. We denote by $\sgame{k}{\mA,\mB}$ the game, played between $\1$ and $\2$, in which in each inning $n\in\w$ $\1$ chooses $A_n\in\mA$ as $\2$ responds with $B_n\subset A_n$ such that $|B_n|\le k$ and $\2$ wins if $\bigcup_{n\in\w}B_n\in\mB$ ($\1$ wins otherwise).
\end{defn}

Again, we note that the classical case here is for $k=1$ in Definition \ref{k_game} with a small change: $\2$ is required to pick exactly one element. Just like with the selection principles, for the two cases we are going to discuss in this paper, this change actually makes no difference.

\begin{defn}\label{fin_game}
	Let $\mA,\mB$ be families of sets. We denote by $\sgame{fin}{\mA,\mB}$ the game, played between $\1$ and $\2$, in which in each inning $n\in\w$ $\1$ chooses $A_n\in\mA$ as $\2$ responds with $B_n\subset A_n$ finite and $\2$ wins if $\bigcup_{n\in\w}B_n\in\mB$ ($\1$ wins otherwise).
\end{defn}

It is worth mentioning here another variation of topological games that have been studied throughout the years (see e.g. \cite{Garcia-Ferreira1995} and \cite{Aurichi2018}). It goes as it follows:

	\begin{defn}
		Let $\mA,\mB$ be families of sets and $f\in\NN^\w$. We denote by $\sgame{f}{\mA,\mB}$ the game, played between $\1$ and $\2$, in which in each inning $n\in\w$ $\1$ chooses $A_n\in\mA$ as $\2$ responds with $B_n\subset A_n$ such that $|B_n|\le f(n)$ and $\2$ wins if $\bigcup_{n\in\w}B_n\in\mB$ ($\1$ wins otherwise).
\end{defn}

As with the classical selection principles, new associated games naturally arise from the new selection principles.
\begin{defn}\label{bnd_game}
	Let $\mA,\mB$ be families of sets. We denote by $\gbound(\mA,\mB)$ the game, played between $\1$ and $\2$, in which in each inning $n\in\w$ $\1$ chooses $A_n\in\mA$ as $\2$ responds with $B_n\subset A_n$ finite and $\2$ wins if there is an $k\in\w$ such that,
	\begin{enumerate}
		\item [a.] $\bigcup_{n\in\w}B_n\in\mB$;
		\item [b.] $|B_n|\le k$ for every $n\in\w$.
	\end{enumerate}
	Otherwise, we say that $\1$ wins.
\end{defn}

Given a game $\G$ played between $\1$ and $\2$, we denote the assertion ``$\1\ (\2)$ has a winning in $\G$'' by $\1\ (\2)\wins\G$ and its negation by $\1\ (\2)\doesntwin\G$.

 \begin{defn}
		Two games $\G_1$ and $\G_2$ are equivalent if the two following assertions hold.
		\begin{itemize}
			\item[(a)] $\1\wins\G_1\iff\1\wins\G_2$;
			\item[(b)]$\2\wins\G_1\iff\2\wins\G_2$.
	\end{itemize}
\end{defn}

As with the usual selection principles, we immediately have:

\begin{prop}\label{classical_nA_imp_sel_bnd}
	Given $\mA$ and $\mB$ families of sets,
	\[
	\1\centernot{\wins}{\gbound(\mA,\mB)}\implies {\sbound(\mA,\mB)}.
	\] 
\end{prop}

The following result will be useful in some arguments.
\begin{lemma}\label{LEMMA_bnd_game}
	Suppose $\sigma$ is a winning strategy for $\2$ in $\gbound(\mA,\mB)$. Then, for every $r\in{^{<\w}\mA}$ there is an $s\in{^{<\w}\mA}$ and an $m\in\NN$ such that $|\sigma(r^\smallfrown s^\smallfrown t)|\le m$ for every $t\in{^{<\w}\mA}$.
\end{lemma}
\begin{proof}
	Suppose our thesis is false and let $r\in{^{<\w}\mA}$ be the sequence witnessing this assertion. Then there is an $s_1\in{^{<\w}\mA}$ such that $|\sigma(r^\smallfrown s_1)|>1$. Again, we may pick an $s_2\in{^{<\w}\mA}$ such that $|\sigma(r^\smallfrown s_1^\smallfrown s_2)|> 2$. Suppose we have picked $\set{s_i:i\le n}$ such that $|\sigma(r^\smallfrown s_1^\smallfrown\cdots^\smallfrown s_k)|> k$ for every $k\le n$. Then we may pick $s_{n+1}\in{^{<\w}\mA}$ such that $|\sigma(r^\smallfrown s_1^\smallfrown\cdots^\smallfrown s_n^\smallfrown s_{n+1})|> n+1$. We have just defined a sequence $\seq{s_n:n\in\NN}$ such that $|\sigma(r^\smallfrown s_1^\smallfrown\cdots^\smallfrown s_n)|> n$ for every $n\in\NN$, a contradiction to the fact that $\sigma$ is a winning strategy in $\gbound(\mA,\mB)$.
\end{proof}

Now, even though Proposition \ref{tight_eqv0} tell us that $\sbound(\Omega_p,\Omega_p)$ is not really a new selection principle, the same cannot be said about the game associated to this principle. In order to prove this, let us first characterize the new game in terms of the already known tightness games:

\begin{thm}\label{Atightbnd_iff_Atightk}
	$\1$ has a winning strategy in $\gbound(\W_p,\W_p)$ if, and only if, $\1$ has a winning strategy in $\sgame{k}{\W_p,\W_p}$ for every $k\in\NN$.
\end{thm}
The idea behind the proof of Theorem \ref{Atightbnd_iff_Atightk} is that, in $\gbound(\W_p,\W_p)$, $\1$ may pretend, at every inning, that the game just started without losing any relevant information, because, in view of Fact \ref{triv_fact0}, the finite set of points $\2$ have chosen thus far is irrelevant to the winning criteria. 

So $\1$ may start the game playing with a winning strategy in the game $\sgame{1}{\W_p,\W_p}$ and, if $\2$ chooses more than one point (say, $k$ points), $\1$ may just pretend the game restarted and then proceed to play with a winning strategy in the game $\sgame{k}{\W_p,\W_p}$. If $\2$ wants to win, he must eventually stop raising the amount of points he chooses, so from that moment on he will be playing against a winning strategy for $\1$ in some $\sgame{k}{\W_p,\W_p}$, and will, therefore, lose the game.

Formally:    
\begin{proof}[Proof of Theorem \ref{Atightbnd_iff_Atightk}]
	Given $k\in\NN$, the implication
	\[
	\1\wins\gbound(\W_p,\W_p)\implies \1\wins\sgame{k}{\W_p,\W_p}
	\]
	is obvious.
	
	So, suppose that for each $k\in\NN$ there is a winning strategy $\gamma_k$ for $\1$ in $\sgame{k}{\W_p,\W_p}$. Then we construct a strategy $\gamma$ for $\1$ in $\gbound(\W_p,\W_p)$ as it follows. First, let $\gamma(\seq{\,})=\gamma_1(\seq{\,})$. If $\2$ chooses $B_0\subset \gamma_1(\seq{\,})$ with $|B_0|\le 1$, then we let $\gamma(\seq{B_0})= \gamma_1(\seq{B_0})$. Otherwise, if $|B_0|=k_0$ for some $k_0>1$, let $\gamma(\seq{B_0})=\gamma_{k_0}(\seq{\,})$. In general, suppose $\gamma$ is defined up to $\seq{B_i:i\le n}$ and that, for each $m\le n$, $\gamma(\seq{B_i:i\le m})=\gamma_{k_m}(\seq{B_i:l_m<i\le m})$, for some $l_m\le m$ and $k_m\in\NN$. If $\2$ chooses $B_{n+1}\subset \gamma(\seq{B_i:i\le n})$ such that $|B_{n+1}|\le k_n$, then we simply put
	\[ 
		\gamma(\seq{B_i:i\le n}^\smallfrown B_{n+1})=\gamma_{k_n}(\seq{B_i:l_n<i\le n}^\smallfrown B_{n+1}).
	\] 
	Otherwise, if $|B_{n+1}|=k_{n+1}>k_n$, we let $\gamma(\seq{B_i:i\le n}^\smallfrown B_{n+1})=\gamma_{k_{n+1}}(\seq{\,})$.
	
	Suppose $\2$ plays $\seq{B_n:n\in\w}$ against $\gamma$ in such a way that, for every $n\in\w$, $|B_n|\le k$ for some (minimal) $k\in\NN$. Then there must be an (also minimal) $l\in\w$ such that $|B_l| = k$ and $|B_n| \le k$ for every $n\ge l$. Then, by the construction presented here, $\seq{B_n: n\ge l}$ is a play against $\gamma_k$. Finally, since each one of the $\gamma_n$'s are winning strategies for $\1$, we may apply Fact \ref{triv_fact0} to $\bigcup_{n\in\w}B_n$ and conclude that if $p\in \overline{\bigcup_{n\in\w}B_n}$, then $p\in\overline{\bigcup_{n\ge l}B_n}$, which would contradict the fact that $\gamma_k$ is a winning strategy in $\sgame{k}{\W_p, \W_p}$. It follows that $\gamma$ is indeed a winning strategy for $\1$ in $\gbound(\W_p,\W_p)$.
\end{proof}

\begin{cor}
	If $\ssel{1}{\Omega_p, \Omega_p}$ does not hold, then $\1$ has a winning strategy in $\gbound(\W_p,\W_p)$. 
\end{cor}

The following result shows us that there is an $f\in\NN^\w$ such that $\gbound(\W_p,\W_p)$ is not equivalent to $\sgame{f}{\W_p,\W_p}$. 
\begin{prop}[\cite{Garcia-Ferreira1995}, Example 3.7; \cite{Aurichi2018}, Example 3.5]
	There is a space $X$ with a point $p$ on which $\2\wins\sgame{f}{\W_p,\W_p}$ for any $f\in\NN^\w$ unbounded, but $\ssel{1}{\Omega_p, \Omega_p}$ fails.
\end{prop}

\begin{cor}\label{TIGHT_gf_not_eqv_gbnd}
	There is a space $X$ with a point $p$ on which $\2\wins\sgame{f}{\W_p,\W_p}$ (in particular, $\2\wins\gfin(\W_p,\W_p)$) and $\1\wins\gbound(\W_p,\W_p)$. 
\end{cor}


On the other hand, to show that $\gbound(\W_p, \W_p)$ is not equivalent to $\sgame{k}{\W_p,\W_p}$ for any $k\in\NN$, we will use the following result:

\begin{prop}[\cite{Aurichi2018}, Example 3.6]\label{ex_aurichi2018}
	For each $k\in\NN$ there is a countable space $X_k$ with only one non-isolated point $p_k$ on which $\1\wins\sgame{k}{\Omega_{p_k},\Omega_{p_k}}$ and $\2\wins\sgame{k+1}{\Omega_{p_k},\Omega_{p_k}}$.
\end{prop}

\begin{cor}\label{bnd_tightgame_ex2}
	For each $k\in\NN$ there is a countable space $X_k$ with only one non-isolated point ${p_k}$ on which $\1\wins\sgame{k}{\Omega_{p_k},\Omega_{p_k}}$, and $\2\wins\gbound(\Omega_{p_k},\Omega_{p_k})$.
\end{cor}

We note that Proposition \ref{ex_aurichi2018} gives us examples on which, for each $k\in\w$, $\2\wins\gbound(\Omega_{p_k},\Omega_{p_k})$. But we concluded this because $\2\wins\sgame{k+1}{\Omega_{p_k},\Omega_{p_k}}$. As the following theorem tells us, this was no coincidence.

\begin{thm}\label{Btightbnd_iff_Btightk}
	$\2$ has a winning strategy in $\gbound(\W_p,\W_p)$ if, and only if, there is an $m\in\NN$ such that $\2$ has a winning strategy in $\sgame{m}{\W_p,\W_p}$.
\end{thm}
\begin{proof}
	It is clear that if $\2$ has a winning strategy in $\sgame{m}{\W_p,\W_p}$ for some $m\in\NN$, then $\2$ has a winning strategy in $\gbound(\W_p,\W_p)$. 
	
	So, suppose $\2$ has a winning strategy $\sigma$ in $\gbound(\W_p,\W_p)$. Without loss of generality, we may assume that $\1$ plays only with sets $A\in\W_p$ such that $p\not\in\overline{\{a\}}$ for every $a\in A$. Let $s\in{^{<\w}\W_p}$ and $m\in\NN$ be as in Lemma \ref{LEMMA_bnd_game} for $r=\seq{\,}$. Then we define a strategy $\sigma_m$ for $\2$ in $\sgame{m}{\W_p,\W_p}$ as it follows: for each $t\in{^{<\w}\W_p}$, let $\sigma_m(t)=\sigma(s^\smallfrown t)$. 
	
	To see that this is a winning strategy, let $\seq{A_n:n\in\w}$ be a sequence of elements of $\W_p$. By construction, $|\sigma_m(A_0, \dotsc A_k)|\le m$ for every $k\in \w$. Also, since $\sigma$ is a winning strategy, $p\in\overline{\sigma(s\upharpoonright 1)\cup \cdots\cup\sigma(s)\cup\left(\bigcup_{k\in\w}\sigma(s^\smallfrown \seq{A_0, \dotsc, A_k})\right)}$. Finally, if we apply Fact \ref{triv_fact0} to the set $\sigma(s\upharpoonright 1)\cup \cdots\cup\sigma(s)\cup\left(\bigcup_{k\in\w}\sigma(s^\smallfrown \seq{A_0, \dotsc, A_k})\right)$, we conclude that $p\in \overline{\bigcup_{k\in\w}\sigma(s^\smallfrown \seq{A_0, \dotsc, A_k})}=\overline{\bigcup_{k\in\w}\sigma_m(\seq{A_0, \dotsc, A_k})}$, and the proof is complete.
\end{proof}

We note that the characterizations presented in Theorems \ref{Atightbnd_iff_Atightk} and \ref{Btightbnd_iff_Btightk} would still hold if we replace ``$\W_p$'' with ``$\D$'' (the set of dense subsets of a given space), because the key argument used there was that, except for some trivial cases, we can ignore finite innings of the game to check the winning criteria. The same thing cannot be said about $\gbound(\mO,\mO)$, because if the game is played over a compact space, for instance, $\2$ may win in the very first inning - but, on the other hand, $\1$ has a winning strategy in $\sgame{k}{\mO,\mO}$ over $2^\w$ for every $k\in\NN$.

So now we turn our attention to covering games:

\begin{prop}\label{bnd_cover_games}
	In every compact space, $\2\wins\gbound(\mO,\mO)$, but $\1\wins\sgame{1}{\mO,\mO}$ over $2^\w$.\\
	Moreover, $\2\wins\gfin(\mO,\mO)$ over every $\sigma$-compact space, but $\1\wins\gbound(\mO,\mO)$ over $\RR$.
\end{prop}

Now suppose $X$ is a space with a compact subset $K$ such that, for every $V\supset K$ open, $\2$ has a winning strategy in $\sgame{1}{\mO,\mO}$ over the complement $X\setminus V$. Clearly, this implies that $\2$ has a winning strategy over $X$ in $\gbound(\mO,\mO)$. What is surprising, though, is that the converse actually holds if $X$ is regular. To prove this, however, we take a step back to define some other variations of the classical selection principles and games.

\section{The ``$\modd\fin$'' and ``$\modd 1$'' variations}\label{SEC_mod1_modfin}
Consider the following simple variations of the classical selection principles, with their respective associated games.

\begin{defn}\label{modfin_sel}
	Let $f\in {\NN}^\w$, and $\mA,\mB$ be families of sets. We say that $\sel{f}{\fin}{\mA,\mB}$ holds if, for every sequence $\seq{A_n:n\in\w}$ of elements of $\mA$, there is a sequence $\seq{B_n:n\in\w}$, such that, 
	\begin{enumerate}
		\item [a.] $B_n\subset A_n$ is finite for every $n\in\w$;
		\item [b.] $\bigcup_{n\in\w}B_n\in\mB$;
		\item [c.] $\set{n\in\w: |B_n|>f(n)}$ is finite.
	\end{enumerate}
	When there is a $k\in\NN$ with $f\equiv k$ we simply write $\sel{k}{\fin}{\mA,\mB}$ instead of $\sel{f}{\fin}{\mA,\mB}$. 
	
	We then define the property $\sel{f}{1}{\mA,\mB}$ as $\sel{f}{\fin}{\mA,\mB}$ with condition (c) replaced by ``$\set{n\in\w: |B_n|>f(n)}\subset\{0\}$'', that is, ``$|B_n|\le f(n)$ for every $n\ge 1$''.  
\end{defn}

\begin{defn}\label{modfin_game}
	Let $f\in {\NN}^\w$, and $\mA,\mB$ be families of sets. We denote by $\game{f}{\fin}{\mA,\mB}$ the game, played between $\1$ and $\2$, in which in each inning $n\in\w$ $\1$ chooses $A_n\in\mA$ as $\2$ responds with $B_n\subset A_n$ finite and $\2$ wins if,
	\begin{enumerate}
		\item [a.] $\bigcup_{n\in\w}B_n\in\mB$;
		\item [b.] $\set{n\in\w: |B_n|>f(n)}$ is finite.
	\end{enumerate}
	When there is a $k\in\NN$ with $f\equiv k$ we simply write $\game{f}{\fin}{\mA,\mB}$ as $\game{k}{\fin}{\mA,\mB}$.

	We then define the game $\game{f}{1}{\mA,\mB}$ as $\game{f}{\fin}{\mA,\mB}$ with condition (b) replaced by ``$\set{n\in\w: |B_n|>f(n)}\subset\{0\}$'' (that is, in other to have a chance of winning the game, $\2$ may choose more elements then $f$ allows only in the first inning).  
\end{defn}

Again, as with the usual selection principles, we also have here:

\begin{prop}\label{classical_nA_imp_sel_mod}
	Let $f\in{\NN}^\w$, and $\mA,\mB$ be families of sets. Then the following implications hold
	\begin{itemize}
		\item $\1\centernot{\wins}{\game{f}{\fin}{\mA,\mB}}\impls {\sel{f}{\fin}{\mA,\mB}}$;
		\item $\1\centernot{\wins}{\game{f}{1}{\mA,\mB}}\impls {\sel{f}{1}{\mA,\mB}}$.
	\end{itemize} 
\end{prop}

In the tightness case, the new selection principles and games collapse to the classical ones:

\begin{prop}\label{tight_eqv1}
	Let $(X,\tau)$ be a topological space and $p\in X$. Then the following properties are equivalent:
	\begin{itemize}
		\item[(1)] $\ssel{1}{\Omega_p, \Omega_p}$;
		\item[(2)] $\ssel{k}{\Omega_p, \Omega_p}$, for every $k\in\NN$;
		\item[(3)] $\sbound(\Omega_p, \Omega_p)$;
		\item[(4)] $\sel{1}{\fin}{\Omega_p, \Omega_p}$;
		\item[(5)]$\sel{1}{{1}}{\Omega_p, \Omega_p}$.
	\end{itemize}
\end{prop}
\begin{proof}
	Clearly $(1)\impls(4)$ and $(1)\impls(5)$. On the other hand, $(4)\impls (3)$ and $(5)\impls (3)$, so the result follows from Proposition \ref{tight_eqv0}.
\end{proof}

\begin{prop}\label{tightgame_eqv0}
	Let $f\in\NN^\w$. Then the following games are equivalent:
	\begin{itemize}
		\item [(a)] $\sgame{f}{\W_p, \W_p}$;
		\item [(b)] $\game{f}{1}{\W_p,\W_p}$;
		\item [(c)] $\game{f}{\fin}{\W_p,\W_p}$.
	\end{itemize}
\end{prop}
\begin{proof}
	We will show the result for $f\equiv 1$ (the general case is analogous).
	The implications
	\[	
	\1\wins{\game{1}{\fin}{\Omega_p, \Omega_p}}\implies\1\wins{\game{1}{1}{\Omega_p, \Omega_p}}\implies\1\wins\sgame{1}{\Omega_p,\Omega_p}
	\]
	\[
	\2\wins\sgame{1}{\Omega_p,\Omega_p}\implies \2\wins{\game{1}{1}{\Omega_p, \Omega_p}}\implies \2\wins{\game{1}{\fin}{\Omega_p, \Omega_p}}
	\]
	are clear.
	
	Suppose there is a winning strategy $\gamma_1$ for $\1$ in $\sgame{1}{\Omega_p,\Omega_p}$. For each sequence $s\in\dom(\gamma_1)$, let $A_s=\gamma_1(s)$ and then fix a choice function $f_s\colon{[A_s]^{<\w}}\to A_s$ (that is, $f_s(F)\in F$ for every $F\subset A_s$ finite). Now, consider the following strategy $\gamma$ for $\1$ in ${\game{1}{\fin}{\Omega_p, \Omega_p}}$:
	\begin{itemize}
		\item Let $\gamma(\seq{\,})=A_{\seq{\,}}$;
		\item After $\2$ chooses 
		$B_0\subset A_{\seq{\,}}$, let 
		\[
		\gamma(\seq{B_0})=A_{\seq{f_{\seq{\,}}(B_0)}};
		\]
		\item After $\2$ chooses $B_1\subset A_{\seq{f_{\seq{\,}}(B_0)}}$, let 
		\[
		\gamma(\seq{B_0, B_1})=A_{\seq{f_{\seq{\,}}(B_0),f_{\seq{B_0}}(B_1) }};
		\]
		\item After $\2$ chooses $B_2\subset A_{\seq{f_{\seq{\,}}(B_0),f_{\seq{B_0}}(B_1) }}$, let 
		\[
		\gamma(\seq{B_0, B_1, B_2})=A_{\seq{f_{\seq{\,}}(B_0),f_{\seq{B_0}}(B_1), f_{\seq{B_0, B_1}}(B_2) }};
		\]
		\item (and so on).
	\end{itemize}
	Note that if we assume that $\seq{B_n:n\in\w}$ is a winning play of $\2$ against $\gamma$, then $\left(\bigcup_{n\in\w}B_n\right)\setminus\set{f_{\seq{\,}}(B_0), f_{\seq{B_0}}(B_1), f_{\seq{B_0, B_1}}(B_2), \dotsc}$ is contained in the finitely many responses of $\2$ in which he chose more than one point, hence it is finite. But since $\gamma_1$ is a winning strategy, $\bigcup_{n\in\w}B_n$ satisfies the hypothesis of Fact \ref{triv_fact0}, so $\seq{f_{\seq{\,}}(B_0), f_{\seq{B_0}}(B_1), f_{\seq{B_0, B_1}}(B_2), \dotsc}$ is a winning play for $\2$ against $\gamma_1$, a contradiction.
	
	Finally, suppose there is a winning strategy $\sigma$ for $\2$ in ${\game{1}{\fin}{\Omega_p, \Omega_p}}$ (we may assume that $\sigma$ always tells $\2$ to choose nonempty subsets). For each $s\in{^{<\w}\Omega_p}$, let $B_s=\sigma(s)$. If there is an $x\in B_s$ such that $p\in \overline{\{x\}}$, fix $b_s=x$. Otherwise, fix any $b_s\in B_s$. Naturally, we define the strategy $\sigma_1$ for $\2$ in $\sgame{1}{\Omega_p,\Omega_p}$ as $\sigma_1(s)=b_s$ for every $s\in{^{<\w}\Omega_p}$. \\
	Now, suppose $\seq{A_n:n\in\w}$ is played by $\1$ in $\sgame{1}{\Omega_p,\Omega_p}$ and let $\seq{B_n:n\in\w}$ and $\seq{b_n:n\in\w}$ be $\sigma$'s and $\sigma_1$'s, respectively, responses to this play. Since $\sigma$ is a winning strategy,
	\begin{enumerate}
		\item [a.] $B=\bigcup_{n\in\w}B_n\in\Omega_p$;
		\item [b.] $\set{k\in\w: |B_k|>1}$ is finite.
	\end{enumerate}
	Then we have two possibilities:
	\begin{itemize}
		\item There is an $x\in B$ such that $p\in\overline{\{x\}}$: in this case, there is an $n\in\w$ such that $p\in\overline{\{b_n\}}$, and so $\seq{b_n:n\in\w}$ is a winning play.
		\item There is no $x\in B$ such that $p\in\overline{\{x\}}$: Then we apply Fact \ref{triv_fact0} to the set $B$ to conclude that $p\in\overline{\set{b_n:n\in\w}}$, hence $\seq{b_n:n\in\w}$ is a winning play.
	\end{itemize}
	It follows that $\sigma_1$ is a winning strategy.
\end{proof}


This is not the case, however, when we consider $\mA=\mB=\mO$, for instance. Note that Proposition \ref{covering_ex} still holds if we replace ``$\sbound(\mO,\mO)$'' by ``$\sel{1}{\fin}{\mO,\mO}$'' or ``$\sel{1}{1}{\mO,\mO}$''. This is no coincidence, as we will see later. But first, consider the following auxiliary results.

\begin{prop}\label{eqv3}
	For every $f\in {\NN}^\w$, 
	\[
	\sel{f}{1}{\mO,\mO}\iff \sel{f}{\fin}{\mO,\mO}.
	\]
\end{prop}
\begin{proof}
	The implication
	\[
		\sel{f}{1}{\mO,\mO}\implies \sel{f}{\fin}{\mO,\mO}
	\] 
	is clear.\\
	Now, suppose $\sel{f}{\fin}{\mO,\mO}$ holds and let $\seq{\mU_n:n\in\w}$ be a sequence of open covers. Let $\seq{\mV_n:n\in\w}$ be the sequence of open covers defined by
	\[
		\mV_n=\mU_0\wedge\cdots\wedge\mU_n.
	\] 
	Since $\sel{f}{\fin}{\mO,\mO}$ holds, there is a sequence $\seq{F_n: n\in\w}$ and a finite $N\subset\w$ such that
	\begin{enumerate}
		\item [a.] $F_n\subset \mV_n$ is finite and therefore, for each $V\in F_n$, $V=U^V_0\cap\cdots\cap U^V_n$ with $U^V_i\in\mU_i$;
		\item [b.] $\bigcup_{n\in\w}F_n\in\mO$;
		\item [c.] $\set{k\in\w: |F_k|>f(k)}= N$.
	\end{enumerate} 
	Let $n_{\max}=\max N$ and $G=\bigcup_{n\le n_{\max}}F_n$. For each $V\in G$ there is a $k_V\in\w$ such that $V=U^V_0\cap\cdots\cap U^V_{k_V}$, so let $U_V=U^V_0$ and $G_0=\set{U_V:V\in G}$. For $0<n\le n_{\max}$, let $G_n=\set{U_n}$ for any $U_n\in\mU_n$. For $n>n_{\max}$, let $G_n=\set{U^V_n:V\in F_n}$. Then 
	\begin{enumerate}
		\item $G_0$ is finite;
		\item $|G_n|=1$, if $0<n\le n_{\max}$;
		\item $|G_n|\le |F_n|$, if $n_{\max}\le n$.
	\end{enumerate} 
	therefore, 
	\begin{enumerate}
		\item [a.] $G_n\subset \mU_n$ for every $n\in\w$;
		\item [b.] $\bigcup_{n\in\w}G_n\in\mO$;
		\item [c.] $\set{n\in\w: |G_n|>f(n)}\subset 1$.
	\end{enumerate} 
It follows that $\sel{f}{1}{\mO,\mO}$ holds.
\end{proof}

\begin{prop}\label{selection_eqv2}
	For all $k\in\NN$ and $f\in {\NN}^\w$:
	\[\sel{1}{1}{\mO,\mO} \iff \sel{k}{1}{\mO,\mO} \iff \sel{f}{1}{\mO,\mO}.\]
\end{prop}
\begin{proof}
	Fix a space $X$. The implications 
	\[\sel{1}{1}{\mO,\mO} \impls \sel{k}{1}{\mO,\mO} \impls \sel{f}{1}{\mO,\mO}\]
	are clear, so suppose $\sel{f}{1}{\mO,\mO}$ holds and let $\seq{\mU_n:n\in\w}$ be a sequence of open covers of $X$. Then we recursively define a new sequence of open covers $\seq{\mW_n:n\in\w}$ as it follows: First, let $\mW_0=\mU_0$. Then we let:
	\begin{itemize}
		\item $\mW_1=\bigwedge_{i=1}^{i=1+f(1)}\mU_i$;
		\item $\mW_2=\bigwedge_{i=2+f(1)}^{i=2+f(1)+f(2)}\mU_i$;		
		\item $\mW_3=\bigwedge_{i=3+f(1)+f(2)}^{i=3+f(1)+f(2)+f(3)}\mU_i$;
		\item and so on.
	\end{itemize} 
	If we apply property $\sel{f}{1}{\mO,\mO}$ to $\seq{\mW_n:n\in\w}$, then we clearly can find a sequence $\seq{\mV_n:n\in\w}$ such that $\bigcup_{n\in\w}\mV_n\in\mO$, $\mV_0\subset \mU_0$ is finite and, for each $n>0$, $\mV_n\subset\mU_n$ and $|\mV_n|\le 1$. Therefore, $\sel{1}{1}{\mO,\mO}$ holds.
\end{proof}

About the covering games, we note that Proposition \ref{bnd_cover_games} still holds if we replace ``$\gbound(\mO,\mO)$'' by ``$\game{1}{\fin}{\mO,\mO}$'' or ``$\game{1}{1}{\mO,\mO}$''. Again, this is no coincidence. But before looking further into this matter, consider the following lemma.

\begin{lemma}\label{LEMMA_gbnd_cover_A_char}
	Let $X$ be a space. Then for every $\mU_0\in\mO$, $\1$ has a winning strategy $\gamma$ in $\game{1}{1}{\mO,\mO}$ such that $\gamma(\seq{\,})=\mU_0$ if, and only if, for every $k\in\NN$ there is a winning strategy $\gamma_k$ for $\1$ in $\game{k}{1}{\mO,\mO}$ with $\gamma_k(\seq{\,})=\mU_0$.
\end{lemma}

In order to prove Lemma \ref{LEMMA_gbnd_cover_A_char} we will evoke the following result.
\begin{thm}[\cite{Crone2019}, Proof of Corollary 2.4]\label{Crone2019A}
	For every $f\in\NN^\w$, $\1$ has a winning strategy in $\sgame{f}{\mO,\mO}$ if, and only if, $\1$ has a winning strategy in $\sgame{1}{\mO,\mO}$.
\end{thm}

\begin{proof}[Proof of Lemma \ref{LEMMA_gbnd_cover_A_char}]
	Let $\mU_0\in\mO$, suppose there is a winning strategy $\gamma$ for $\1$ in $\game{1}{1}{\mO,\mO}$ such that $\gamma(\seq{\,})=\mU_0$ and fix $k\in\NN$. Note that $\1\wins\sgame{1}{\mO,\mO}$ over $X\setminus\bigcup F$ for every $F\subset \mU_0$ finite, which implies (by Theorem \ref{Crone2019A}) that there is a winning strategy $\gamma_k^F$ for $\1$ in $\sgame{k}{\mO,\mO}$ over $X\setminus\bigcup F$. Now, consider the following strategy:
	\begin{itemize}
		\item First, let $\gamma_k(\seq{\,})=\mU_0$;
		\item If $\2$ then chooses $F_0\subset \gamma_k(\seq{\,})$ finite, let 
		\[
		\gamma_k(\seq{F_0})=\set{\text{V open}: V\cap \left(X\setminus\bigcup F_0\right)\in \gamma_k^{F_0}(\seq{\,})}\cup\left\{\bigcup F_0\right\};
		\]
		\item If $\2$ then chooses $F_1=\{V_1\}\subset \gamma_k(\seq{F_0})$, let
		\[
		\gamma_k(\seq{F_0, \mV_1}) = \set{\text{V open}: V\cap \left(X\setminus\bigcup F_0\right)\in \gamma_k^{F_0}(\seq{F_1})}\cup\left\{\bigcup F_0\right\}
		\]
		(we are assuming here that $\2$ will not choose $V_1=\bigcup F_0$, since its points were already covered in the first inning);
		\item And so on.
	\end{itemize}
	Clearly, $\gamma_k$ has the desired properties.
	
	The other implication is obvious. 
\end{proof}

Now, the following theorem will help us show one of the main results of this paper.

\begin{thm}\label{g_char}
	Let $X$ be a regular space. Then $\2\wins\game{1}{1}{\mO,\mO}$ if, and only if, there is a compact set $K\subset X$ such that, for every open set $V\supset K$, $\2\wins\sgame{1}{\mO,\mO}$ over $X\setminus V$.
\end{thm}
\begin{proof}
	Suppose there is a compact set $K\subset X$ such that, for every open set $V\supset K$, there is a winning strategy $\sigma^V_1$ for $\2$ in $\sgame{1}{\mO,\mO}$ over $X\setminus V$. Then we define the following strategy $\sigma$ for $\2$ in $\game{1}{1}{\mO,\mO}$:
	\begin{itemize}
		\item If $\1$ starts with $\mU_0\in\mO$, let $\sigma(\mU_0)$ be a finite subcover for $K$ and let $V=\bigcup\sigma(\mU_0)$.
		\item After that, if $\seq{\mU_0, \dotsc, \mU_n}$ is played by $\1$, let $\sigma(\seq{\mU_0, \dotsc, \mU_n})=\sigma^V_1(\seq{\mU_1, \dotsc, \mU_n})$.
	\end{itemize}
	Then, clearly, $\sigma$ is a winning strategy.
	
	Now, suppose $\sigma$ is a winning strategy for $\2$ in $\game{1}{1}{\mO,\mO}$.
		\begin{claim}\label{compact_menger}The set 
			\[
			K =\bigcap_{\mU\in\mO}\overline{\bigcup\sigma(\seq{\mU})}
			\] 
			is compact.
		\end{claim} 
		\begin{proof}
			Indeed, let $\mC$ be an open cover for $K$ and for each $x\in K$, let $U_x\in\mC$ be such that $x\in U_x$. Since $X$ is regular, for every $x\in K$ there is an open set $V_x$ such that $x\in V_x\subset\overline{V_x}\subset U_x$. On the other hand, for each $x\in X\setminus K$ we consider an open set $V_x$ such that $x\in V_x$ and $\overline{V_x}\cap K =\emptyset$ (because $K$ is closed and $X$ is regular). Now, let $\mU=\set{V_x:x\in X}\in\mO$. In this case, note that 
			\[
			K\subset\overline{\bigcup\sigma(\seq{\mU})}.
			\]
			Consider $\mA=\set{V_x: (x\in K)\wedge(V_x\in \sigma(\seq{\mU}))}=\set{V_{x_1}, \dotsc,V_{x_n}}$, with $x_1, \dotsc, x_n\in K$. Then $K\subset\overline{\bigcup\mA}$. Finally, note that $\set{U_{x_1}, \dotsc, U_{x_n}}\subset \mC$ is a finite subcover of $K$.
		\end{proof}
	Now, let $V$ be an open set containing $K$. Note that since $\2\wins\game{1}{1}{\mO,\mO}$, $X$ is Lindel\"of, and since $X\setminus V$ is closed, $X\setminus V$ is Lindel\"of. With that in mind, if we consider the open cover $\set{X\setminus\overline{\bigcup\sigma(\seq{\mU})}: \mU\in\mO}$ of $X\setminus V$, we may obtain a countable subcover $\set{X\setminus\overline{\bigcup\sigma(\seq{\mU_n})}: n\in\NN}$. If $\mV$ is an open cover of $X\setminus V$, let $\mV'=\mV\cup\{V\}\in\mO$ and fix an enumeration $\set{p_n:n\in\NN}$ of the prime numbers of $\w$. Now we have everything at hand to define a winning strategy $\sigma^V_1$ for $\2$ in $\sgame{1}{\mO,\mO}$ over $X\setminus V$:
		\[
			\sigma^V_1(\seq{\mV_0, \dotsc, \mV_n}) =\begin{cases}
			\sigma(\seq{\mU_k, \mV'_{p_k^1}, \dotsc, \mV'_{p_k^m}})\setminus \{V\}, \text{ if $n=p_k^m$ for some $k,m\in\NN$};\\
			\{U_n\} \text{ with $U_n\in\mV_n$ (anyone!), otherwise.}  
			\end{cases}  
		\]
	To show that $\sigma^V_1$ is, indeed, winning, let $y\in X\setminus V$ and consider $\seq{\mV_n:n\in\w}$ as any play from $\1$ in $\sgame{1}{\mO,\mO}$ over $X\setminus V$. Since $\set{X\setminus\overline{\bigcup\sigma(\seq{\mU_n})}: n\in\NN}$ covers $X\setminus V$, $y\notin \overline{\bigcup\sigma(\seq{\mU_k})}$ for some $k\in\NN$. But since $\sigma$ is a winning strategy in $\game{1}{1}{\mO,\mO}$, $y$ must be covered by some of $\sigma$'s responses to $\1$'s play $\seq{\mU_k}^\smallfrown\seq{\mV'_{p_k^n}: n\in\NN}$, so $\sigma^V_1$ covers $y$ and, therefore, is a winning strategy.
\end{proof}

But how does this new selection principles relate to the ``bounded versions'' presented here? As it turns out, in a very simple way.

\begin{thm}\label{bnd_coversel_char}
	$\sbound(\mO,\mO)$ holds if, and only if, $\sel{1}{1}{\mO,\mO}$ holds.
\end{thm}
\begin{proof}
	The implication
	\[
	\sel{1}{1}{\mO,\mO}\implies \sbound(\mO,\mO)
	\]
	is clear, so suppose $\sbound(\mO,\mO)$ holds. We define $f\in\NN^\w$ as $f(n)=n+1$. Now, since for every $k\in\w$ the set $\set{n\in\w:k>f(n)}$ is finite, the result follows from the fact that $\sbound(\mO,\mO)$ holds if, and only if, $\sel{f}{\fin}{\mO,\mO}$ holds and by Propositions \ref{eqv3} and \ref{selection_eqv2}.
\end{proof}

Regarding the games, $\gbound(\mO,\mO)$ is equivalent (over Hausdorff spaces) to $\game{1}{1}{\mO,\mO}$. We show this assertion in the following theorems.

\begin{thm}\label{1_char_bnd_covergame}
	$\1$ has a winning strategy in $\gbound(\mO,\mO)$ if, and only if, $\1$ has a winning strategy in $\game{1}{1}{\mO,\mO}$.
\end{thm}
The idea behind the proof of Theorem \ref{1_char_bnd_covergame} is similar to the one presented in the proof of Theorem \ref{Atightbnd_iff_Atightk}. 

The main difference here is that $\1$ cannot just pretend the game restarted at any inning without losing important information, because $\2$ have indeed covered a portion of the space thus far. Lemma \ref{LEMMA_gbnd_cover_A_char}, however, gives us instructions of how she can switch between strategies pretending the game is back to the second inning without losing this important information.

Formally speaking:

\begin{proof}[Proof of Theorem \ref{1_char_bnd_covergame}]
	The implication
	\[
	\1\wins\gbound(\mO,\mO) \implies \1\wins\game{1}{1}{\mO,\mO}
	\]
	is clear.
	
	So, suppose $\gamma$ is a winning strategy for $\1$ in $\game{1}{1}{\mO,\mO}$ and let $\gamma_k$, for each $k\in\NN$ be as in Lemma \ref{LEMMA_gbnd_cover_A_char} with $\mU_0=\gamma(\seq{\,})$ (that is, such that $\gamma_k(\seq{\,})=\mU_0$ for every $k\in\NN$). We will assume that $\gamma$ and $\gamma_k$, for every $k\in\NN$, tell $\1$ to play refinements of $\mU_0$ in every turn. Now consider the following strategy:
	\begin{itemize}
		\item First, let $\tilde{\gamma}(\seq{\,})=\mU_0$.
		\item If $\2$ chooses $F_0\subset\mU_0$ with $|F_0|=k_0$ for some $k_0\in\NN$, let 
		\[
		\tilde{\gamma}(\seq{F_0})=\gamma_{k_0}(\seq{F_0});
		\]
		\item If $\2$ chooses $F_1\subset \tilde{\gamma}(\seq{F_0})$ such that $|F_1|\le k_0$, then let 
		\[
		\tilde{\gamma}(\seq{F_0, F_1})=\gamma_{k_0}(\seq{F_0, F_1}),
		\]
		otherwise, if $|F_1|=k_1> k_0$, then for each $V\in F_1$ fix $U_V\in\mU_0$ such that $V\subset U_V$ and let 
		\[
		\tilde{\gamma}(\seq{F_0, F_1})=\gamma_{k_1}(\seq{F_1'}),
		\]
		with $F_1'=\set{U_V:V\in F_1}\cup F_0$;
		\item And so on.
	\end{itemize}
	Clearly, $\tilde{\gamma}$ is a winning strategy for $\1$ in $\gbound(\mO,\mO)$.
\end{proof}

\begin{thm}\label{2_char_bnd_covergame}
	Let $X$ be a Hausdorff space. Then $\2\wins \gbound(\mO,\mO)$ if, and only if, $\2\wins \game{1}{1}{\mO,\mO}$.
\end{thm}

In order to prove Theorem \ref{2_char_bnd_covergame} we will use the following theorem:

\begin{thm}[\cite{Crone2019}, Corollary 2.4]\label{Crone2019} 
	If $X$ is a Hausdorff space, then, for every $k\in\NN$ and $f\in\NN^\w$, the following games are equivalent over $X$:
	\begin{itemize}
		\item $\sgame{1}{\mO,\mO}$;
		\item $\sgame{k}{\mO,\mO}$;
		\item $\sgame{f}{\mO,\mO}$.
	\end{itemize}
\end{thm}

\begin{proof}[Proof of Theorem \ref{2_char_bnd_covergame}]
	The implication
	\[
	\2\wins \game{1}{1}{\mO,\mO}\implies \2\wins \gbound(\mO,\mO)
	\]
	is clear, so let $\sigma$ be a winning strategy for $\2$ in $\gbound(\mO,\mO)$. 
	
	Note that we can assume that $\1$ plays always with refinements of her first cover played in the game. For each $\mU\in\mO$, let $s_\mU\in{^{<\w}\mO}$ and $m_\mU\in\NN$ be as in Lemma \ref{LEMMA_bnd_game} for $r=\seq{\mU}$. Now, fixed $\mU\in\mO$, we fix, for each $U\in\bigcup_{k\in\dom (s_\mU)+1}\sigma(s_\mU\upharpoonright k)$, $V_U\in \mU$ such that $U\subset V_U$. Then we let 
	\[
	\tilde{\sigma}(\seq{\mU})=\set{V_U:U\in\bigcup_{k\in\dom (s_\mU)+1}\sigma(s_\mU\upharpoonright k)}.
	\]
	Note that, by our hypothesis, $\2\wins \sgame{m_\mU}{\mO,\mO}$ over $X\setminus\bigcup\tilde{\sigma}(\seq{\mU})$ for each $\mU\in\mO$, so it follows from Theorem \ref{Crone2019} that there is a winning strategy $\sigma_\mU$ for $\2$ in $\sgame{1}{\mO,\mO}$ over $X\setminus\bigcup\tilde{\sigma}(\seq{\mU})$ for each $\mU\in\mO$. Then we define, for each $s\in{^{<\w}\mO}$, 
	\[
	\tilde{\sigma}(\seq{\mU}^\smallfrown s)=\sigma_\mU(s),
	\] 
	and it is clear that the strategy $\tilde{\sigma}$ we have just defined is a winning strategy for $\2$ in $\game{1}{1}{\mO,\mO}$.
\end{proof}

\begin{cor}
	Let $X$ be Hausdorff space. Then, for every $f\in \NN^\w$, 
	\[
	\2\wins \gbound(\mO,\mO)\iff\2\wins \game{f}{\fin}{\mO,\mO}\iff\2\wins \game{f}{1}{\mO,\mO}
	\]
\end{cor}

\begin{cor}
	The games $\gbound(\mO,\mO)$ and  $\game{1}{1}{\mO,\mO}$ are equivalent over every Hausdorff space.
\end{cor}
 
 And finally:
 \begin{thm}\label{gbnd_cover_char}
 	Let $X$ be a regular space. Then $\2\wins\gbound(\mO,\mO)$ if, and only if, there is a compact set $K\subset X$ such that, for every open set $V\supset K$, $\2\wins\sgame{1}{\mO,\mO}$ over $X\setminus V$.
 \end{thm}
\begin{proof}
		It follows directly from Theorems \ref{g_char} and \ref{2_char_bnd_covergame}. 
\end{proof}

Theorem \ref{gbnd_cover_char} is useful to characterize even stricter sets on metric spaces. To see this, consider the following classical result.

\begin{thm}[Telg\'arsky (\cite{Telgarsky1983}); Galvin (\cite{Galvin1978})]
	Let $X$ be a space in which every point is a $G_\delta$ set. Then $\2\wins\sgame{1}{\mO,\mO}$ if, and only if, $X$ is countable.
\end{thm}

\begin{cor}
	Let $X$ be a regular space such that every compact subset is a $G_\delta$ subset (e.g. a metrizable space). Then $\2\wins\sgame{bnd}{\mO,\mO}$ if, and only if, there is a compact set $K\subset X$ and a countable set $N\subset X$  such that $X=K\cup N$.
\end{cor}

\section{The analogous to Pawlikowski's and Hurewicz's results}\label{SEC_Paw_Hur} 
Now, recall the following classical theorems.
\begin{thm}[Hurewicz]\label{hurewicz}
	$\ssel{fin}{\mO,\mO}\iff\1\centernot\wins\gfin(\mO,\mO)$
\end{thm} 

\begin{thm}[Pawlikowski]\label{pawlikowski}
	$\ssel{1}{\mO,\mO}\iff\1\centernot\wins\sgame{1}{\mO,\mO}$
\end{thm} 

As it turns out, our previous results can help us show an analogous theorem here, in the ``bounded'' variation. The following proof is heavily inspired by the simplified proof of Theorem \ref{pawlikowski} that can be seen, for instance, at \cite{Szewczak}.  

\begin{thm}\label{classical_nA_eqv_sbnd_cover}
	$\sbound(\mO,\mO)\iff\1\doesntwin\gbound(\mO,\mO)$
\end{thm}
\begin{proof}
	Implication $\1\doesntwin\gbound(\mO,\mO)\impls\sbound(\mO,\mO)$
	is clear by Proposition \ref{classical_nA_imp_sel_bnd}.
	
	To show the reverse implication, by Proposition \ref{bnd_coversel_char} and Theorem \ref{1_char_bnd_covergame},  it suffices to show that 
	\[
	\sel{1}{1}{\mO,\mO}\implies \1\doesntwin\game{1}{\fin}{\mO,\mO},
	\]
	so suppose $\sel{1}{1}{\mO,\mO}$ holds and let $\gamma$ be a strategy for $\1$ in $\game{1}{\fin}{\mO,\mO}$. For simplicity's sake, in the rest of this proof we will write ``$\{V\}$'' simply as ``$V$''.
	
	We then recursively define the following strategy $\tilde\gamma$ for $\1$ in $\gfin(\mO,\mO)$ and function $f$:
	
	We first let $\tilde\gamma(\seq{\,})=\gamma(\seq{\,})$. Then, for each $V_0\in\tilde\gamma(\seq{\,})$,
	
	\[
	f(\seq{V_0})=V_0,
	\]
	and, for each finite $\mF_0\subset\tilde\gamma(\seq{\,})$, let
	\[
	f(\seq{\mF_0})=\set{f(\seq{V_0}):V_0\in\mF_0}=\mF_0.
	\]
	
	Suppose $\mF_0$ was chosen by $\2$. Then we let
	
	\[
	\tilde\gamma(\seq{\mF_0})=\gamma(\seq{f(\seq{\mF_0})})\wedge\bigwedge_{V_0\in\mF_0}\gamma(\seq{f(\seq{V_0})}).
	\]
	
	Now, for each $V_0\in\mF_0$ and $V_1\in\tilde\gamma(\seq{\mF_0})$, define
	
	\[
	f(\seq{V_0,V_1})=V\in\gamma(\seq{f(\seq{V_0})}) \text{ such that $V\supset V_1$;}
	\]
	\[
	f(\seq{\mF_0,V_1})=V\in\gamma(\seq{f(\seq{\mF_0})})\text{ such that $V\supset V_1$;}
	\]
	and, for each finite $\mF_1\subset\tilde\gamma(\seq{\mF_0})$,
	\[
	f(\seq{\mF_0,\mF_1})=\set{	f(\seq{\mF_0,V_1}):V_1\in\mF_1}.
	\]
	
	Suppose $\mF_1$ is then chosen by $\2$. Then we let
	
	\begin{align*}
		\tilde\gamma(\seq{\mF_0,\mF_1})=&\gamma(\seq{f(\seq{\mF_0}), f(\seq{\mF_0,\mF_1})})\wedge \left(\bigwedge_{V_1\in\mF_1}\gamma(\seq{f(\seq{\mF_0}), f(\seq{,\mF_0, V_1})})\right)\wedge \\
		\wedge&\left(\bigwedge_{V_0\in\mF_0}\bigwedge_{V_1\in\mF_1}\gamma(\seq{f(\seq{V_0}), f(\seq{V_0, V_1})})\right),
	\end{align*}
	
	for each $V_0\in\mF_0$, $V_1\in\mF_1$ and $V_2\in\tilde\gamma(\seq{\mF_0,\mF_1})$,
	
	\begin{align*}
		f(\seq{V_0,V_1, V_2})&=V\in\gamma(\seq{f(\seq{V_0}), f(\seq{V_0, V_1})}) \text{ such that $V\supset V_2$;}\\
		f(\seq{\mF_0,V_1, V_2})&=V\in\gamma(\seq{f(\seq{\mF_0}), f(\seq{\mF_0, V_1})}) \text{ such that $V\supset V_2$;} \\
		f(\seq{\mF_0,\mF_1, V_2})&=V\in\gamma(\seq{f(\seq{\mF_0}), f(\seq{\mF_0, \mF_1})}) \text{ such that $V\supset V_2$,}
	\end{align*}	
	and, for each finite $\mF_2\subset\tilde\gamma(\seq{\mF_0,\mF_1})$,
	\[
	f(\seq{\mF_0,\mF_1, \mF_2})=\set{	f(\seq{\mF_0,\mF_1, V_2}):V_2\in\mF_2}.
	\]
	
	Now let us look at the general case. Suppose we have defined $\tilde\gamma$ and $f$ up to $s\in\dom\tilde\gamma$ in such a way that, for every $k\le |s|$:
	
	\begin{align*}
		\tilde\gamma(s\restrict k)=&\gamma(\seq{f(s\restrict 1), \dotsc, f(s\restrict k)})\wedge \\
		\wedge&\left(\bigwedge_{V_{k-1}\in s(k-1)}\gamma(\seq{f(s\restrict 1)), \dotsc, f(s\restrict k-1) , f((s\restrict k-1)^\smallfrown V_{k-1})})\right)\wedge \\
		&\vdots \\
		\wedge&\left(\bigwedge_{V_0\in s(0)}\bigwedge_{V_1\in s(1)}\cdots\bigwedge_{V_{k-1}\in s(k-1)}\gamma(\seq{f(\seq{V_0}), \dotsc, f(\seq{V_0, \dotsc, V_{k-1}})})\right),
	\end{align*}
	
	for all $V_0\in s(0), \dotsc, V_{k-1}\in s(k-1), V_k\in \tilde\gamma(s\restrict k)$:
	
	\begin{align*}
		f(\seq{V_0, \dotsc, V_{k-1}, V_k})=&V\in\gamma(\seq{f(\seq{V_0}), \dotsc, f(\seq{V_0, \dotsc, V_{k-1}})}) \\
		&\text{ such that $V\supset V_k$;}\\
		f(\seq{s(0), V_1, \dotsc, V_{k-1}, V_k})=& V\in\gamma\seq{f(\seq{s(0)}), f(\seq{s(0), V_1}), \dotsc, f(\seq{s(0), \dotsc, V_{k-1}})}) \\
		&\text{ such that $V\supset V_k$;}\\
		&\vdots\\
		f((s\restrict k)^\smallfrown V_k)=& V\in \gamma(\seq{f(s\restrict 1), \dotsc, f(s\restrict k)}) \text{ such that $V\supset V_k$,}
	\end{align*}
	
	and for every $\mF_k\subset \tilde{\gamma}(s\restrict k)$,
	
	\[
	f((s\restrict k)^\smallfrown \mF_k)=\set{f((s\restrict k)^\smallfrown V_k):V_k\in\mF_k}.
	\]
	
	Then if $\2$ chooses $\mF_n\subset\tilde\gamma(s)$ we let
	
	\begin{align*}
		\tilde\gamma({s^\smallfrown\mF_n})=&\gamma(\seq{f(s\restrict 1), \dotsc, f(s), f(s^\smallfrown\mF_n)})\wedge \\
		\wedge&\left(\bigwedge_{V_{n}\in \mF_n}\gamma(\seq{f(s\restrict 1)), \dotsc, f(s) , f(s^\smallfrown V_{n})}\right)\wedge \\
		&\vdots \\
		\wedge&\left(\bigwedge_{V_0\in s(0)}\cdots\bigwedge_{V_{n-1}\in s(n-1)}\bigwedge_{V_{n}\in \mF_n}\gamma(\seq{f(\seq{V_0}), \dotsc, f(\seq{V_0, \dotsc, V_{n}})})\right),
	\end{align*}
	
	for all $V_0\in s(0), \dotsc, V_{n-1}\in s(n-1), V_n\in\mF_n,  V_{n+1}\in\tilde\gamma(s^\smallfrown\mF_n)$:
	
	\begin{align*}
		f(\seq{V_0, \dotsc, V_{n}, V_{n+1}})=&V\in\gamma(\seq{f(\seq{V_0}), \dotsc, f(\seq{V_0, \dotsc, V_{n}})}) \\
		&\text{ such that $V\supset V_{n+1}$;}\\
		f(\seq{s(0), V_1, \dotsc, V_{n}, V_{n+1}})=& V\in\gamma\seq{f(\seq{s(0)}), f(\seq{s(0), V_1}), \dotsc, f(\seq{s(0), \dotsc, V_{n}})}) \\
		&\text{ such that $V\supset V_{n+1}$;}\\
		&\vdots\\
		f(s^\smallfrown \mF_n^\smallfrown V_{n+1})=& V\in \gamma(\seq{f(s\restrict 1), \dotsc, f(s)}) \text{ such that $V\supset V_{n+1}$,}
	\end{align*}
	
	and for every $\mF_{n+1}\subset \tilde{\gamma}(s^\smallfrown \mF_{n})$,
	
	\[
	f(s^\smallfrown \mF_n^\smallfrown \mF_{n+1})=\set{f(s^\smallfrown \mF_n^\smallfrown V_{n+1}):V_{n+1}\in\mF_{n+1}},
	\]
	
	so the recursion is complete.
	
	Now, since $\sel{1}{1}{\mO,\mO}$ holds, $\ssel{fin}{\mO,\mO}$ holds and, by Theorem \ref{hurewicz}, $\tilde\gamma$ is not a winning strategy. Moreover, $\2$ can play a sequence $\seq{\mF_n:n\in\w}$ against $\tilde\gamma$ such that $\bigcup_{n\ge m}\bigcup\mF_n=X$ for every $m\in\w$ (to see this, just note that if $\1\doesntwin\gfin(\mO,\mO)$ over $X$, then $\1\doesntwin\gfin(\mO,\mO)$ over $X\times\w$).
	
	\begin{claim}\label{CLAIM_classical_nA_eqv_sbnd_cover}
		There is an $N\in\w$ and a choice of $V_n\in\mF_n$ for each $n\ge N$ such that $\left(\bigcup_{n\le N}\mF_n\right)\cup\left(\bigcup_{n> N}V_n\right)=X$.
	\end{claim}
	\begin{proof}
		For each $n\in\w$ let 
		\[
		\mW_n=\set{V^{k_0}\cap\cdots\cap V^{k_n}: V^{k_i}\in\mF_{k_i} \text{ for all $i\le n$ and $k_0<k_1<\cdots<k_n$}}.
		\]
		Note that $\mW_n$ is an open cover for every $n\in\w$. Then, since $\sel{1}{1}{\mO,\mO}$ holds, we can find $\{V^{k_0}, \dotsc V^{k_m}\}\subset\mW_0$ with $V^{k_i}\in\mF_{k_i}$ for each $i\le k_m$ and a single $U_n\in\mW_n$ for each $n>0$ such that $\left(\bigcup_{i\le m}V^{k_i}\right)\cup\left(\bigcup_{n>0}U_n\right)=X$.
		
		Let $N=\max\set{k_i:i\le m}$. Now from each $U_n$ we can pick a $V_{l_n}\in\mF_{l_n}$ such that $U_n\subset V_{l_n}$ and $l_n\neq l_i$ for all $i<n$. Then if we pick any $V_k\in\mF_k$ when $k\neq l_n$ for every $n>0$, the proof is complete.  
	\end{proof}
	
	Now we define a winning play for $\2$ against $\gamma$ as it follows. For each inning $n\le N$, let $\2$ respond to $\gamma$ with $f(\seq{\mF_i:i\le n})$. Then, for each $n\ge N$, let $\2$ respond to $\gamma$ with $f(\seq{\mF_i:i\le N}^\smallfrown\seq{V_j:j\le n})$. It follows from the definition of $f$ and from Claim \ref{CLAIM_classical_nA_eqv_sbnd_cover} that $\2$ wins this play in $\game{1}{\fin}{\mO,\mO}$, hence $\gamma$ is not a winning strategy.
\end{proof}



One may wonder if Theorem \ref{gbnd_cover_char} still holds if we replace ``$\2\wins\gbound(\mO,\mO)$'' by ``$\sbound(\mO,\mO)$'' and ``$\2\wins\sgame{1}{\mO,\mO}$'' by ``$\ssel{1}{\mO,\mO}$''. The answer is yes. But to show that, let us first take another step back and define yet another variation of the classical selection principles.

\begin{defn}
	Let $(X, \tau)$ be a topological space. We say the property $\sels{1}{1}{\mO,\mO}$ holds if for every open cover $\mU$ there is a $\mV\subset\mU$ finite such that $\ssel{1}{\mO,\mO}$ holds over $X\setminus \bigcup \mV$.
\end{defn}




At first glance, this new variation may seem stronger than $\sel{1}{1}{\mO,\mO}$. However, we will show later that they are equivalent selection principles. This will be useful because:

\begin{prop}
	Let $X$ be a regular space. Then $\sels{1}{1}{\mO,\mO}$ holds if, and only if, there is a compact set $K\subset X$ such that, for every open set $V\supset K$, $\ssel{1}{\mO,\mO}$ holds over $X\setminus V$.
\end{prop}
\begin{proof}
	Analogous to the proof of Theorem \ref{g_char}.
\end{proof}

\begin{prop}
	$\sels{1}{1}{\mO,\mO}\iff \1\centernot\wins\game{1}{1}{\mO,\mO}$
\end{prop}
\begin{proof}
	Suppose $\sels{1}{1}{\mO,\mO}$ holds and let $\gamma$ be a strategy for $\1$ in $\game{1}{1}{\mO,\mO}$. Then there is a $\mV\subset\gamma(\seq{\,})$ such that $\ssel{1}{\mO,\mO}$ holds over $X\setminus \bigcup \mV$, so it follows from Theorem \ref{pawlikowski} that $\gamma$ cannot be a winning strategy.
	
	On the other hand, suppose $\sels{1}{1}{\mO,\mO}$ fails. Then there is an open cover $\mU$ such that $\ssel{1}{\mO,\mO}$ fails over $X\setminus\bigcup\mV$ for every finite $\mV\subset\mU$. Let $\gamma(\seq{\,})=\mU$ and, if $\2$ responds with a finite $\mV\subset\mU$, then $\1$ can simply use the sequence of open covers of $X\setminus \bigcup\mV$ that witnesses that $\ssel{1}{\mO,\mO}$ fails to win the game.
\end{proof}

\begin{cor}\label{COR_sbnd_eqv_selsmod1}
	$\sels{1}{1}{\mO,\mO}\iff \sel{1}{1}{\mO,\mO}\iff\sbound(\mO,\mO)$.
\end{cor}

\begin{cor}\label{sbnd_eqv_thereisKcompact}
	Let $X$ be a regular space. Then $\sbound(\mO,\mO)$ holds if, and only if, there is a compact set $K\subset X$ such that, for every open set $V\supset K$, $\ssel{1}{\mO,\mO}$ holds over $X\setminus V$.
\end{cor}

With the help of Corollary \ref{sbnd_eqv_thereisKcompact} we can even characterize some metrizable spaces. We just need to consider the following result from Fremlin and Miller.

\begin{thm}[\cite{Miller1988}, Theorem 1]
	Given a metrizable space $(X,\tau)$, $\ssel{1}{\mO,\mO}$ holds if, and only if, $X$ has strong measure zero with respect to every metric which gives topology $\tau$.
\end{thm}

\begin{cor}\label{sbnd_eqv_thereisKcompact&smz}
	Let $(X,\tau)$ be a metrizable space. Then $\sbound(\mO,\mO)$ holds if, and only if, there is a compact set $K\subset X$ and a set $N\subset X$ that is strong measure zero with respect to every metric that gives topology $\tau$ such that $X=K\cup N$.
\end{cor}

\section{The dual game}\label{SEC_dual}
Let us recall a classical topological game:

\begin{defn}\label{DEFN_point-open}
	The {\bf point-open game} is the following game played between $\1$ and $\2$ over a space $X$: in each inning $n\in\w$ $\1$ chooses $x_n\in X$ and $\2$ responds with an open neighborhood $V_n$ of $x_n$. $\1$ wins the game if $\bigcup_{n\in\w}V_n=X$ and $\2$ wins otherwise.
\end{defn}

We are interested in this game because it is known to have a strong relation with one of the games studied here:

\begin{defn}
	Two games $\G_1$ and $\G_2$ are {\bf dual} if the two following assertions hold.
	\begin{itemize}
		\item[(a)] $\1\wins\G_1\iff\2\wins\G_2$;
		\item[(b)]$\2\wins\G_1\iff\1\wins\G_2$.
	\end{itemize}
\end{defn}

\begin{thm}[\cite{Galvin1978}, Galvin]\label{THM_po_roth_dual}
	The point-open game is dual to $\sgame{1}{\mO,\mO}$.
\end{thm}

Our goal here is to find a duality similar to \ref{THM_po_roth_dual} for $\sgame{bnd}{\mO,\mO}$, that is, to find a variation of \ref{DEFN_point-open} that is dual to $\sgame{bnd}{\mO,\mO}$. One natural variation is the game in which in each inning $n\in\w$ $\1$ is allowed to choose finitely many points (instead of just one) and $\2$ has to cover those points with an open set (this game is known as ``finite-open game''). It can be easily checked, however, that this variation is actually equivalent to the point-open game (in fact, Telg\'arsky introduced this game and proved this equivalence in \cite{Telgarsky1983}). So, consider the following. 

\begin{defn}\label{DEFN_compact&point-open}
	Given a space $X$, we denote by $\G(X)$ the following game played between $\1$ and $\2$: in the first inning, $\1$ chooses a compact set $K_0$ and $\2$ responds with $V_0\supset K_0$ open. Then in each inning $n>0$ $\1$ chooses $x_n\in X$ and $\2$ responds with an open neighborhood $V_n$ of $x_n$. $\1$ wins the game if $\bigcup_{n\in\w}V_n=X$ and $\2$ wins otherwise.
\end{defn}

In this case, our duality naturally rises as a simple translation of Theorems \ref{gbnd_cover_char}, \ref{classical_nA_eqv_sbnd_cover} and Corollary \ref{sbnd_eqv_thereisKcompact}:

\begin{thm}\label{gbnd_duality}
	For every topological space:
	\begin{itemize}
		\item[(a)] If $\1\wins\gbound(\mO,\mO)$, then $\2\wins\G(X)$;
		\item[(b)] If $\1\wins\G(X)$, then $\2\wins\gbound(\mO,\mO)$.
	\end{itemize}
	Moreover, if $X$ is a regular space:
	\begin{itemize}
		\item[(c)] If $\2\wins\gbound(\mO,\mO)$, then $\1\wins\G(X)$;
		\item[(d)] If $\2\wins\G(X)$, then $\1\wins\gbound(\mO,\mO)$.
	\end{itemize}
\end{thm}
\begin{proof}
	Assertions (a) and (b) can be easily checked. Assertion (c) follows directly from Theorems \ref{gbnd_cover_char} and \ref{THM_po_roth_dual}. 
	
	Now, suppose $\2\wins\G(X)$. Then for every $K\subset X$ compact there is a $V\supset K$ open such that $\2$ has a winning strategy in the point open game over $X\setminus V$. By \ref{THM_po_roth_dual}, this implies that for every $K\subset X$ compact there is a $V\supset K$ open such that $\1\wins \sgame{1}{\mO,\mO}$ over $X\setminus V$. By Theorem \ref{pawlikowski}, this means that for every $K\subset X$ compact there is a $V\supset K$ open such that $\ssel{1}{\mO,\mO}$ fails over $X\setminus V$. Since $X$ is regular, by Corollary \ref{sbnd_eqv_thereisKcompact}, this is equivalent to $\ssel{bnd}{\mO,\mO}$ failing over $X$, which, by Theorem \ref{classical_nA_eqv_sbnd_cover}, is equivalent to $\1\wins\sgame{bnd}{\mO,\mO}$, as we wanted to prove. 
\end{proof}

We then end this section showing that the assumption of $X$ being a regular space is actually required in the proof of (c) and (d) in Theorem \ref{gbnd_duality}:

\begin{prop}\label{reg_required_covergame}
	There is a Hausdorff and non-regular space $X$ such that $\2\wins\sgame{bnd}{\mO,\mO}$, but $\2\wins\G(X)$.
\end{prop}
\begin{proof}
	Let $(X,\tau)$ be a Hausdorff space such that $\2\wins\sgame{bnd}{\mO,\mO}$ and $\2$ has a winning strategy in the point-open game (for instance, $2^{\w}$) and consider a new topology $\rho$ over $X$ that additionally makes every countable set closed. 
	
	Clearly, $\2$ still has a winning strategy in the point-open game (or, equivalently, the finite-open game) over the new topological space. Moreover, it is easy to see that, in the new topology, $K\subset X$ is compact if, and only if, $K$ is finite. So it follows that $\2\wins\G((X,\rho))$.
	
	However, $\2$ still has a winning strategy in $\sgame{bnd}{\mO,\mO}$ over the new topological space $(X,\rho)$. To see that, we first let $\set{A_k:k\in\w}$ be a partition of the odd numbers in $\w$ made by infinite subsets such that $\min{A_i}<\min{A_j}$ when $i<j$ and let $\sigma$ be a winning strategy for $\2$ in $\game{1}{1}{\mO,\mO}$ over the original topological space (that exists, because $(X,\rho)$ remains Hausdorff and by Theorem \ref{2_char_bnd_covergame}). In the new space, we may assume that $\1$ chooses only covers with open sets of the form $U\setminus C$, with $U\in\tau$ and $C$ countable. Given $\mU$ open cover of $(X,\rho)$ with said form we fix, for each $U\in\mU$, $U'$ as the open set from the original topology such that $U=U'\setminus C$ for some $C$ countable. Then we let, for each open cover $\mU$ of $(X,\rho)$ with said form, 
	\[
	\mU'=\set{U'\in\tau:U\in\mU}.
	\] 
	Now we define a strategy $\tilde{\sigma}$ as it follows: 
	\begin{itemize}
		\item in the first inning $(n=0)$, if $\1$ chooses $\mU_0$, let \[\tilde{\sigma}(\seq{\mU_0})=\set{U\in\mU_0: U'\in \sigma(\seq{\mU_0'})}.\]
		Note that $\bigcup\sigma(\seq{\mU_0'})\setminus \bigcup\tilde{\sigma}(\seq{\mU_0})$ is countable. Then we let $\tilde{\sigma}$ cover these points in the odd innings of the set $A_0$;
		\item if in the next even inning $(n=2)$, $\1$ chooses $\mU_2$, let \[\tilde{\sigma}(\seq{\mU_0, \mU_1,\mU_2})=\set{U\in\mU_2: U'\in \sigma(\seq{\mU_0', \mU_2'})}.\]
		Note that $\bigcup\sigma(\seq{\mU_0', \mU_2'})\setminus \bigcup\tilde{\sigma}(\seq{\mU_0, \mU_1, \mU_2})$ is countable. Then we let $\tilde{\sigma}$ cover these points in the odd innings of the set $A_1$;
		\item if in the next even inning $(n=4)$, $\1$ chooses $\mU_4$, let \[\tilde{\sigma}(\seq{\mU_0, \mU_1,\mU_2, \mU_3,\mU_4})=\set{U\in\mU_4: U'\in \sigma(\seq{\mU_0', \mU_2',\mU_4'})}.\]
		Note that $\bigcup\sigma(\seq{\mU_0', \mU_2',\mU_4'})\setminus \bigcup\tilde{\sigma}(\seq{\mU_0, \mU_1, \mU_2,\mU_3,\mU_4})$ is countable. Then we let $\tilde{\sigma}$ cover these points in the odd innings of the set $A_2$;
		\item and so on.
	\end{itemize}
	Clearly, $\tilde{\sigma}$ is a winning strategy in $\sgame{bnd}{\mO,\mO}$ over $(X,\rho)$, and the proof is complete. 
\end{proof}

\begin{cor}\label{reg_required_coversel}
	There is a Hausdorff non-regular space $X$ such that $\sbound(\mO,\mO)$ holds, but for every compact $K\subset X$ there is an open set $V\supset K$ such that $\ssel{1}{\mO,\mO}$ fails over $X\setminus V$.
\end{cor}

\section{Conclusion}\label{SEC_CONC}
The results obtained in this paper can be summarized in the following diagrams (Figure \ref{DIAGRAM_tight} is dedicated to the tightness case and Figure \ref{DIAGRAM_cover} is dedicated to the covering case). Arrows represent implications. The number immediately next to an arrow tells us where is the proof of such implication (if it is not obvious) and the number between parenthesis immediately next to it points out to the counterexample of its converse implication. Indications such as ``Regular'' or ``$T_2$'' next to an arrow tell us that this assumption was required in the specified proof and the number between parenthesis next to this indication points out to the counterexample showing that without said assumption the implication would fail. For simplicity's sake, we will denote ``$\1$'' by ``$\A$'' and ``$\2$'' by ``$\B$''.

With all of that in mind, we quote here some results that show counterexamples to some of the implications in the diagram.  

\begin{prop}[\cite{Gruenhage2006}, Example 2.11; \cite{Aurichi2018}, Example 3.10]\label{TIGHT_AnotG1_BnotGFIN}
	There is a countable space with only one non-isolated point $p$ on which $\1\doesntwin \sgame{1}{\W_p,\W_p}$ and $\2\doesntwin\gfin(\W_p,\W_p)$. 
\end{prop}

\begin{prop}[\cite{Scheepers1997}, pp. 250-251; \cite{Aurichi2018}, Example 2.4]\label{TIGHT_S1_AwinsGFIN}
	There exists a countable space $X$ with only one non-isolated point $p$ on which $\ssel{1}{\W_p,\W_p}$ holds (hence, $\ssel{fin}{\W_p,\W_p}$ holds) and $\1\wins\gfin(\W_p,\W_p)$.
\end{prop}

We denote by $C_p(X)$ the subspace of $\RR^X$ of continuous functions. If $f\in C_p(X)$ is constant and equal to $0$, then we simply denote $f$ by $0$.

\begin{thm}[\cite{Barman2011}, Theorem 3.6]\label{Barman1}
If $X$ is $\sigma$-compact and metrizable, then $\2\wins\gfin(\W_0,\W_0)$ on $C_p(X)$.
\end{thm}

\begin{thm}[\cite{Sakai1988}, Theorem 1]\label{Sakai1}
For every space X, $\ssel{1}{\W_0,\W_0}$ holds over $C_p(X)$ if, and only if, $\ssel{1}{\mO,\mO}$ holds over each finite product of $X$.
\end{thm}

\begin{cor}\label{bnd_tightgame_ex}
Over $C_p(\RR)$: 
\begin{itemize}
	\item[(a)] $\2\wins\gfin(\W_0,\W_0)$;
	\item[(b)] $\ssel{1}{\W_0,\W_0}$ fails.
\end{itemize}
\end{cor}

\begin{prop}[\cite{Telgarsky1983}, Section 7; \cite{Aurichi2013}, Example 3.5]\label{COVER_AnotG1_BnotGFIN}
	There is a space on which $\ssel{1}{\mO,\mO}$ holds (hence, $\1\doesntwin\sgame{1}{\mO,\mO}$), but $\2\doesntwin\gfin(\mO,\mO)$.
\end{prop}

\begin{landscape}
	\centering
\begin{tikzpicture}
\matrix (m) [matrix of nodes, row sep=20mm, column sep=5mm]
{
	$\B\wins\game{1}{1}{\W_p,\W_p}$ & $\B\wins\game{k}{1}{\W_p,\W_p}$& $\exists k\in\NN\B\wins\sgame{k}{\W_p,\W_p}$& &\\
	$\B\uparrow\sgame{1}{\W_p,\W_p}$ & $\B\uparrow\sgame{k}{\W_p,\W_p}$ & $\B\uparrow\gbound(\Omega_p, \Omega_p)$ & $\B\uparrow\gfin(\Omega_p, \Omega_p)$ & $\A\doesntwin\gfin(\W_p,\W_p)$\\
	 &  &  &  & $\ssel{fin}{\W_p,\W_p}$\\
	$\A\doesntwin\sgame{1}{\W_p,\W_p}$ & $\A\doesntwin\sgame{k}{\W_p,\W_p}$ & $\A\doesntwin\gbound(\W_p,\W_p)$ &  & $\sbound(\W_p,\W_p)$\\ 
	$\A\doesntwin\game{1}{1}{\W_p,\W_p}$ & $\A\doesntwin\game{k}{1}{\W_p,\W_p}$& $\exists k\in\NN\A\doesntwin\sgame{k}{\W_p,\W_p}$& & $\ssel{1}{\W_p,\W_p}$ \\ 
};

\draw[->] (m-1-1) to node[auto]{(\ref{ex_aurichi2018})} (m-1-2);
\draw[->] (m-1-2) to node[auto]{(\ref{bnd_tightgame_ex2})} (m-1-3);

\draw[->] (m-2-1) to node[auto]{(\ref{ex_aurichi2018})} (m-2-2);
\draw[->] (m-2-2) to node[auto]{(\ref{bnd_tightgame_ex2})} (m-2-3);
\draw[->] (m-2-3) to node[auto]{(\ref{TIGHT_gf_not_eqv_gbnd})} (m-2-4);
\draw[->] (m-2-4) to node[auto]{(\ref{TIGHT_AnotG1_BnotGFIN})} (m-2-5);

\draw[->] (m-4-1) to node[auto]{(\ref{ex_aurichi2018})} (m-4-2);
\draw[->] (m-4-2) to node[auto]{(\ref{bnd_tightgame_ex2})} (m-4-3);
\draw[->] (m-4-3) to node[auto]{\ref{classical_nA_imp_sel_bnd} (\ref{TIGHT_S1_AwinsGFIN})} (m-4-5);

\draw[->] (m-5-1) to node[auto]{(\ref{ex_aurichi2018})} (m-5-2);
\draw[->] (m-5-2) to node[auto]{(\ref{bnd_tightgame_ex2})} (m-5-3);

\draw[<->] (m-1-1) to node[auto]{\ref{tightgame_eqv0}} (m-2-1);
\draw[<->] (m-1-2) to node[auto]{\ref{tightgame_eqv0}} (m-2-2);
\draw[<->] (m-1-3) to node[auto]{\ref{Btightbnd_iff_Btightk}} (m-2-3);

\draw[->] (m-2-5) to node[auto]{(\ref{TIGHT_S1_AwinsGFIN})} (m-3-5);

\draw[->] (m-4-5) to node[auto, swap]{(\ref{bnd_tightgame_ex})} (m-3-5);

\draw[->] (m-2-1) to node[auto]{(\ref{TIGHT_AnotG1_BnotGFIN})} (m-4-1);
\draw[->] (m-2-2) to node[auto]{(\ref{TIGHT_AnotG1_BnotGFIN})} (m-4-2);
\draw[->] (m-2-3) to node[auto]{(\ref{TIGHT_AnotG1_BnotGFIN})} (m-4-3);

\draw[<->] (m-4-1) to node[auto]{\ref{tightgame_eqv0}} (m-5-1);
\draw[<->] (m-4-2) to node[auto]{\ref{tightgame_eqv0}} (m-5-2);
\draw[<->] (m-4-3) to node[auto]{\ref{Atightbnd_iff_Atightk}} (m-5-3);
\draw[<->] (m-4-5) to node[auto]{\ref{tight_eqv0}} (m-5-5);

\draw[->] (m-4-3) to node[auto,swap]{(\ref{TIGHT_gf_not_eqv_gbnd})} (m-2-5);
\end{tikzpicture}

\captionof{figure}{Tightness case}
\label{DIAGRAM_tight}
\end{landscape}

\begin{landscape}
	\centering
	\begin{tikzpicture}
	\matrix (m) [matrix of nodes, row sep=20mm, column sep=8mm]
	{
		 & & $\B\wins\game{1}{1}{\mO,\mO}$& $\A\wins\G(X)$& &\\
		$\B\uparrow\sgame{1}{\mO,\mO}$ & $\B\uparrow\sgame{k}{\mO,\mO}$ & $\B\uparrow\gbound(\mO, \mO)$ & $\B\uparrow\gfin(\mO, \mO)$ & $\A\doesntwin\gfin(\mO,\mO)$ &\\
		&  &  &  & $\ssel{fin}{\mO,\mO}$ &\\
		$\A\doesntwin\sgame{1}{\mO,\mO}$ & $\A\doesntwin\sgame{k}{\mO,\mO}$ & $\A\doesntwin\gbound(\mO,\mO)$ &  & $\sbound(\mO,\mO)$ & $\sel{1}{1}{\mO,\mO}$\\ 
		 & & $\A\doesntwin\game{1}{1}{\mO,\mO}$ & $\B\doesntwin\G(X)$ & $\ssel{1}{\mO,\mO}$  & $\sels{1}{1}{\mO,\mO}$\\ 
	};

	\draw[<->] (m-2-1) to node[auto]{\cite{Crone2019}} node[auto,swap]{$T_2$} (m-2-2);
	\draw[->] (m-2-2) to node[auto]{(\ref{bnd_cover_games})} (m-2-3);
	\draw[->] (m-2-3) to node[auto]{(\ref{bnd_cover_games})} (m-2-4);
	\draw[->] (m-2-4) to node[auto]{(\ref{COVER_AnotG1_BnotGFIN})} (m-2-5);
	
	\draw[<->] (m-4-1) to node[auto]{\cite{Crone2019}} (m-4-2);
	\draw[->] (m-4-2) to node[auto]{(\ref{bnd_cover_games})} (m-4-3);
	\draw[<->] (m-4-3) to node[auto]{\ref{classical_nA_eqv_sbnd_cover}} (m-4-5);
	\draw[<->] (m-4-5) to node[auto]{\ref{bnd_coversel_char}} (m-4-6);
	
	\draw[<->] (m-5-4) to node[auto,swap]{\ref{gbnd_duality}} node[auto] {Regular (\ref{reg_required_covergame})} (m-5-3);
	
	\draw[<->] (m-1-3) to node[auto]{\ref{gbnd_duality}} node[auto,swap] {Regular (\ref{reg_required_covergame})} (m-1-4) ;
	
	\draw[->] (m-1-3) to node[auto]{\ref{2_char_bnd_covergame}}node[auto,swap] {$T_2$} (m-2-3);
	
	\draw[->] (m-2-3) to node[auto]{} (m-1-3);
	
	\draw[<->] (m-2-5) to node[auto]{\ref{hurewicz}} (m-3-5);
	
	\draw[->] (m-2-5) to node[auto]{} (m-3-5);
	
	\draw[->] (m-2-1) to node[auto]{(\ref{COVER_AnotG1_BnotGFIN})} (m-4-1);
	\draw[->] (m-2-2) to node[auto]{(\ref{COVER_AnotG1_BnotGFIN})} (m-4-2);
	\draw[->] (m-2-3) to node[auto]{(\ref{COVER_AnotG1_BnotGFIN})} (m-4-3);
	\draw[->] (m-4-5) to node[auto,swap]{(\ref{covering_ex})} (m-3-5);
	
	\draw[<->] (m-4-3) to node[auto]{\ref{1_char_bnd_covergame}} (m-5-3);
	\draw[<-] (m-4-5) to node[auto]{(\ref{covering_ex})} (m-5-5);
	
	\draw[<->] (m-4-6) to node[auto]{\ref{COR_sbnd_eqv_selsmod1}} (m-5-6);
	
	\draw[->] (m-4-3) to node[auto,swap]{(\ref{bnd_cover_games})} (m-2-5);
	\end{tikzpicture}
	
	\captionof{figure}{Covering case}
	\label{DIAGRAM_cover}
\end{landscape}

\newpage


In the proof of Theorem \ref{2_char_bnd_covergame} we used the main result of \cite{Crone2019}, which is why we required $X$ to be Hausdorff. So, just like it was done in \cite{Crone2019}, it is only natural to end here with the question:

\begin{prob}\label{PROB_hausdorff}
	Is there a non-Hausdorff space $X$ such that $\2\wins\gbound(\mO,\mO)$, but $\2\doesntwin\game{1}{1}{\mO,\mO}$?
\end{prob}

In fact, it is easy to see that Problem \ref{PROB_hausdorff} is actually equivalent to the problem presented in \cite{Crone2019}:

\begin{prob}\label{PROB_Crone2019}
	Is there a non-Hausdorff space $X$ such that $\2\wins\sgame{k}{\mO,\mO}$ for some $k\in\NN$, but $\2\doesntwin\sgame{1}{\mO,\mO}$?
\end{prob}

\section{Acknowledgements}
We thank Piotr Szewczak and Boaz Tsaban for giving us access to the preliminary notes of \cite{Szewczak} and we also thank Henrique A. Lecco, who made the question that motivated the beginning of this paper.  


\end{document}